\documentclass[11pt,reqno]{amsart}

\usepackage{times}
\usepackage{graphicx}
\usepackage{amssymb}
\usepackage{latexsym}
\usepackage{enumerate}
\usepackage[bookmarks=false]{hyperref}
\calclayout
\theoremstyle{plain}
\numberwithin{equation}{section}
\newtheorem{thm}{Theorem}[section]

\newtheorem{proposition}[thm]{Proposition}

\newtheorem{construction}[thm]{Construction}
\newtheorem{problem}[thm]{Problem}

\begin{document}

\title{The Triangle of Smallest Area Which Circumscribes a Semicircle}
\author{Jun Li}
\address{
School of Science\\
Jiangxi University of Science and Technology\\ Ganzhou\\
341000\\
China.
}
\email{junli323@163.com}

\begin{abstract}
An interesting problem that determine a triangle of smallest area which circumscribes a semicircle is solved. Then a generalized golden right triangles sequence $T_n$ is obtained, and an interesting construction of the maximum generalized golden right triangle $T_2$ is also shown.
\end{abstract}

\date{2016.6}
\maketitle

\section{The Triangle of Smallest Area Which Circumscribes a Semicircle}

In \cite{cte1}, DeTemple showed that the isosceles triangle of smallest perimeter which circumscribes a semicircle is made up of two congruent right triangles which have sides proportional to $(1, \sqrt{\phi}, \phi)$, and the right triangle is known as the Kepler triangle\cite[p. 149]{cte2}\cite{cte3}. 
Here, we'll consider another interesting problem that determine a triangle of smallest area which circumscribes a semicircle.

\begin{figure}[ht]
\centering
\includegraphics[width=1\textwidth]{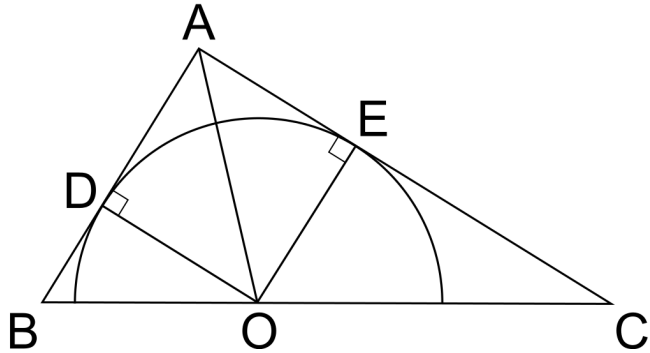}
\caption{The triangle of smallest area which circumscribes a semicircle}
\label{fig:fg1a}
\end{figure}
In Figure \ref{fig:fg1a}, Let a triangle $\triangle{ABC}$ circumscribe a semicircle of radius $R=1$, with the diameter of the semicircle contained in the base $BC$, let $O$ denote the center of the semicircle, $OD$ and $OE$ are both the radii, then let $AB=x$ and $\angle{A}=\theta$, now, our problem is:
\begin{problem}\label{pm}
If the radius $R$ and the two sides $AB$ and $AC$ of $\triangle{ABC}$ can not be three edges of a triangle, determine $\triangle{ABC}$ of smallest area.
\end{problem}
\begin{proposition}\label{prop1}
If $R$, $AB$ and $AC$ can not be three edges of a triangle, then the of smallest area $\triangle{ABC}$ is a right triangle similar to the ($1$, $\phi$,$\sqrt {1+\phi^2}$) triangle which forms half of a golden rectangle\cite{cte8}\cite[p. 274]{cte5}\cite[p. 115]{cte9}.
\end{proposition}
\begin{proof}
If $R$, $AB$ and $AC$ can not be three edges of a triangle, then first, we suppose the radius $R=1$ is the longest segment, thus, we have $AC=R-AB=1-x$, where $0<x<1$, then in $\triangle{ABC}$, we have the area equation (\ref{1})
\begin{equation}\label{1}
S_{ABC}=S_{AOB}+S_{AOC}
\end{equation}
and 
$$
S_{ABC}=\frac{1}{2}{AB}\cdot{AC}\sin{\theta}=\frac{1}{2}x(1-x)\sin{\theta}
$$
$$
S_{AOB}+S_{AOC}=\frac{1}{2}x+\frac{1}{2}(1-x)
$$
then we get 
\begin{equation}\label{1a}
\sin{\theta}=\frac{1}{x(1-x)}>1
\end{equation}
hence, we conclude that the radius $R$ can not be the longest segment among the three segments.

Next, we suppose $AC$ is the longest segment, then we have $AC=R+AB=1+x$,
using the similar method, we can get 
\begin{equation}\label{1b}
\sin{\theta}=\frac{2x+1}{x^2+x}
\end{equation}
since $0< \sin{\theta} \le 1$, we have 
\begin{equation}\label{1c}
0<\frac{2x+1}{x^2+x}\le 1
\end{equation}
and with $x>0$, we conclude that $x \ge \phi$, where $\phi=\frac{1 + \sqrt{5} }{2}$, hence 
\begin{equation}\label{1d}
S_{ABC}=\frac{2x+1}{2}\ge \frac{{\phi}^3}{2}
\end{equation}
The equality in (\ref{1d}) holds if and only if $x=\phi$. Then we get $AB=x=\phi$, $AC=1+x={\phi}^2$, and in (\ref{1b}), we have $\sin{\theta}= 1$, which means $\triangle{ABC}$ is a right triangle having sides proportional to ($1$, $\phi$,$\sqrt {1+\phi^2}$). 
\end{proof}
\section{A sequence of generalized golden right triangles}\label{sec:2}
\begin{figure}[ht]
\centering
\includegraphics[width=1\textwidth]{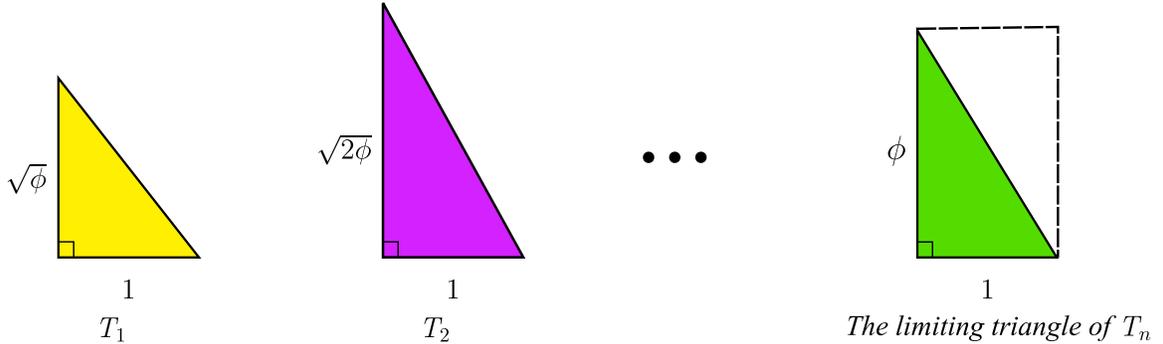}
\caption{A sequence of generalized golden right triangles $T_n$}
\label{fig:fg2}
\end{figure}
There is an identity (\ref{2}) of the golden ratio\cite{cte2} and Fibonacci numbers (see, e.g., \cite[p. 78]{cte5}).
\begin{equation}\label{2}
\phi^{n+1} = F_{n+1}\phi + F_{n}, \quad (n = 0, 1, 2, \ldots)
\end{equation}
If we rewrite (\ref{2}) in the form of (\ref{2a}),
\begin{equation}\label{2a}
1 + {\left( {\sqrt {\frac{\phi {F_{n+1} }}{{{F_{n}}}}} } \right)^2} = {\left( {\sqrt {\frac{{{\phi ^{n+1}}}}{{{F_{n}}}}} } \right)^2}, \quad (n = 1, 2, 3, \ldots)
\end{equation}
we will obtain a right triangles sequence $T_n$ with sides ($1$, $\sqrt {\frac{{{\phi F_{n+1}} }}{{{F_{n}}}}}$, $\sqrt {\frac{{{\phi ^{n+1}}}}{{{F_{n}}}}}$), see Figure \ref{fig:fg2}.
It's easy to see that, the first right triangle $T_1$ with sides ($1$, $\sqrt{\phi}$, $\phi$) is the Kepler triangle whose side lengths are in geometric progression, the second right triangle $T_2$ is a ($1$, $\sqrt{2\phi}$, $\phi\sqrt{\phi}$) triangle, and furthermore, let $n \to +\infty$, we find that the limiting right triangle of $T_n$ is just a ($1$, $\phi$, $\sqrt {1+\phi^2}$) triangle which forms half of a golden rectangle. 

Since there are only two golden (see \cite[p. 73]{cte8}) right triangles (one is the Kepler triangle with sides ($1$, $\sqrt{\phi}$, $\phi$), the other is the ($1$, $\phi$, $\sqrt {1+\phi^2}$) triangle), and they are both special cases of $T_n$, then, we can call $T_n$ a generalized golden right triangles sequence. In addition, we have the following simple area inequality (\ref{2b}) for $T_n$,
\begin{equation}\label{2b}
\triangle{T_1} \le \triangle{T_n} \le \triangle{T_2}
\end{equation}
where $\triangle{T_n}$ denoted as the area of $T_n$, and 
\begin{equation}\label{2c}
\triangle{T_n}=\frac{{\sqrt \phi  }}{2}\sqrt {\frac{{{F_{n + 1}}}}{{{F_n}}}}, \quad (n = 1, 2, 3, \ldots)
\end{equation}
in the inequality (\ref{2b}), we also notice that the second triangle $T_2$ with sides ($1$, $\sqrt{2\phi}$, $\phi\sqrt{\phi}$) is just the maximum area triangle in $T_n$, therefore, we can call $T_2$ the maximum generalized golden right triangle.

Interestingly enough, similar to the construction of the Kepler triangle $T_1$ by first creating a golden rectangle shown in \cite{cte3}, we can also construct the maximum generalized golden right triangle $T_2$ by first constructing a golden rectangle, see Figure \ref{fig:fg2a}.
\begin{figure}[ht]
\centering
\includegraphics[width=0.3\textwidth]{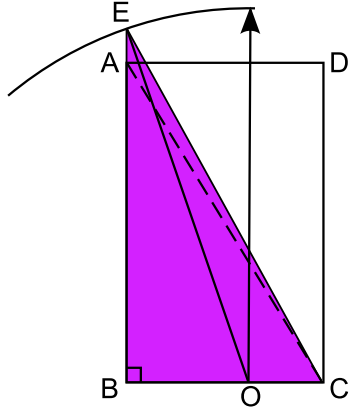}
\caption{An interesting construction of $T_2$}
\label{fig:fg2a}
\end{figure}
\begin{construction}\label{cons2}
A simple and interesting 3-step construction of $T_2$:
\begin{enumerate}[(1)]
\item construct a golden rectangle $ABCD$ with $BC = 1$, $AB = \phi$ (see, e.g., \cite[p. 118]{cte9})
\item construct $O$ dividing $BC$ in the golden ratio and $\frac{BO}{OC}=\phi$
\item draw an arc with the center at $O$ and the radius of length $AC$, cutting the extension of $BA$ at $E$, and join $E$ to $C$
\end{enumerate}
Then $\triangle{EBC}$ is just $T_2$ having sides $(1, \sqrt{2\phi}, \phi\sqrt{\phi})$.
\end{construction}
\begin{proof}
$BO=\frac{BC}{\phi}=\frac{1}{\phi}$, $EO=AC=\sqrt {1+\phi^2}$, thus, $BE = \sqrt{EO^2-BO^2}=\sqrt{2\phi}$.
\end{proof}
Last, back to Figure \ref{fig:fg1a} again, we give a problem to readers:
\begin{problem}\label{pmlst}
If the radius $R$ and the two sides $AB$ and $AC$ of $\triangle{ABC}$ can not be three edges of a triangle, determine $\triangle{ABC}$ of smallest perimeter. And if $\triangle{ABC}$ can not be an acute-angled triangle, determine $\triangle{ABC}$ of smallest perimeter.
\end{problem}



\medskip

\noindent Mathematics Subject Classification (2010).  51M04, 51M15, 11B39

\noindent Keywords.  Smallest Area, Triangle, Semicircle, Golden rectangle, Golden ratio, Fibonacci numbers, Kepler triangle, Golden right triangle, Maximum generalized golden right triangle
\end{document}